\documentclass[10pt]{amsart}

\usepackage[colorlinks]{hyperref}
\usepackage{color,graphicx,shortvrb}
\usepackage[latin 1]{inputenc}

\usepackage[active]{srcltx} 

\usepackage{enumerate}
\usepackage{amssymb}

\newtheorem{theorem}{Theorem}[section]

\newtheorem{corollary}[theorem]{Corollary}

\newtheorem{lemma}[theorem]{Lemma}

\newtheorem{proposition}[theorem]{Proposition}
\newtheorem{remark}[theorem]{Remark}

\def\11{\textbf{$1$}}

\DeclareGraphicsExtensions{.jpg,.pdf,.png,.eps}

\begin{document}

\title[Orthogonal forms and orthogonality preservers revisited]{Orthogonal forms and orthogonality preservers on real function algebras revisited}

\author[A.M. Peralta]{Antonio M. Peralta}

\address{Departamento de An{\'a}lisis Matem{\'a}tico, Facultad de
Ciencias, Universidad de Granada, 18071 Granada, Spain.}
\email{aperalta@ugr.es}

\thanks{Author partially supported by the Spanish Ministry of Economy and Competitiveness project no. MTM2014-58984-P and Junta de Andaluc\'{\i}a grant FQM375.}



\date{}

\begin{abstract} In \cite{GarPe2014} we determine the precise form of a continuous orthogonal form on a commutative real C$^*$-algebra. We also describe the general form of a (not-necessarily continuous) orthogonality preserving linear map between commutative unital real C$^*$-algebras. Among the consequences, we show that every orthogonality preserving linear bijection between commutative unital real C$^*$-algebras is continuous. In this note we revisit these results and their proofs with the idea of filling a gap in the arguments, and to extend the original conclusions.
\end{abstract}

\maketitle
\thispagestyle{empty}

\section{Introduction}

Let $A$ be a real or complex C$^*$-algebra. Elements $a,b$ in $A$ are said to be \emph{orthogonal} (written $a\perp b$) if $a b^*=b^* a=0$. A bilinear form $V: A\times A \to \mathbb{C}$ is said to be \emph{orthogonal} if $V(a,b)=0$ whenever $a\perp b$. In \cite{GarPe2014} we establish a generalization of a celebrated result due to S. Goldstein (see \cite[Theorem 1.10]{Gold}) by proving the following result:

\begin{theorem}\label{t orhtogonal forms real abelian}\cite[Theorem 2.4]{GarPe2014}
Let $V: A\times A \to \mathbb{R}$ be a continuous orthogonal form on a commutative real C$^*$-algebra,
then there exist $\varphi_1$ and $\varphi_2$ in $A^{*}$ satisfying $$V(x,y) = \varphi_1 (x y) + \varphi_2 ( x y^*),$$ for every $x,y\in A.$
\end{theorem}

We recently realized the presence of a ``gap'' affecting some of the technical results given in \cite{GarPe2014}. The concrete difficulties appear in the following arguments: By the Gelfand theory for commutative real C$^*$-algebras, every commutative unital real C$^*$-algebra $A$ is C$^*$-isomorphic (and hence isometric) to a real function algebra of the form $C(K)^{\tau}=\{f\in C(K) : \tau (f) =f\}$, where $K$ is a compact Hausdorff space, $\tau$ is a conjugation (i.e. a conjugate linear isometry of period 2) on $C(K)$ given by $\tau (f) (t) =\overline{f(\sigma(t))}$ ($ t\in K$), and $\sigma : K\to K$ is a topological involution (i.e. a period 2 homeomorphism) (compare \cite[Proposition 5.1.4]{Li}). Let $\sigma: K \to K$ be a topological involution on a compact Hausdorff space $K$. Clearly, the sets $N=\{t\in K : \sigma (t) \neq t\}$ and $F=\{t\in K : \sigma (t) = t\}$ are open and closed subsets of $K$, respectively. It is established in \cite[Lemma 2.1 and its proof]{GarPe2014} that the family,  $\mathcal{F}$, of all open subsets $O\subseteq K$ such that $O\cap \sigma( O)=\emptyset,$ ordered by inclusion is an inductive set, and hence, by Zorn's lemma, there exists an open subset $\mathcal{O}\subset \mathcal{F}$ maximal with respect to the property $\mathcal{O}\cap \sigma (\mathcal{O}) = \emptyset.$ Immediately after \cite[Lemma 2.1]{GarPe2014} it is claimed that, 

\emph{``by the maximality of $\mathcal{O}$, $K=F \stackrel{\circ}{\cup} \mathcal{O} \stackrel{\circ}{\cup} \sigma (\mathcal{O}).$''}\smallskip

Unfortunately, the above equality is not always true. Consider, for example, $K=\mathbb{T}$ the unit sphere of $\mathbb{C}$ and $\sigma: K\to K$, $\sigma( t) =-t$ ($t\in \mathbb{T}$). In this case $F=\emptyset$ and $\mathcal{O} =\{t\in \mathbb{T} : \Im\hbox{m} (t) > 0\}$ is a maximal set in $\mathcal{F}$, but $F \stackrel{\circ}{\cup} \mathcal{O} \stackrel{\circ}{\cup} \sigma (\mathcal{O}) \neq K$.\label{page 1} This gap affects several statements and proofs of technical results in \cite[Sections 2 and 3]{GarPe2014}.\smallskip

We recall that a mapping $T: A\to B$ between real or complex C$^*$-algebras is said to be \emph{orthogonality or disjointness preserving} if $a\perp b$ in $A$ implies $T(a)\perp T(b)$ in $B.$\smallskip

In the second main goal studied in \cite{GarPe2014}, we consider orthogonality preserving linear maps between real function algebras belonging to a special subclass of the category of commutative unital real C$^*$-algebras. The algebras in this particular subclass can be presented as follows: Let $F$ be a closed subspace of a compact Hausdorff space $K$. We denote by $C_r (K) = C_r(K;F)$ the real C$^*$-algebra of all continuous functions $f : K \to \mathbb{C}$ taking real values on $F$. The main result in \cite[\S 3]{GarPe2014} is presented in Theorem 3.2, where we establish a complete description of those linear (not necessarily continuous) orthogonality preserving operators between $C_r(K;F)$-spaces. It should be remarked here that the proof of this result is not affected by the gap commented above. Among the consequences not affected by the difficulties, we obtain the following result.

\begin{theorem}\label{p inverse preserves invertible}\cite[Remark 3.4, proof of Theorem 3.5 and Proposition 3.6]{GarPe2014} In the notation above, let $T: C_r(K_1;F_1)\to C_r(K_2;F_2)$ be an orthogonality preserving linear bijection. Then $T$ is automatically continuous and $T^{-1}$ preserves invertible elements, that is, $T^{-1} (g)$ is invertible whenever $g$ is an invertible element in $C_r(L_2)$. $\hfill\Box$
\end{theorem}

However, it must be remarked that the main goal expected from these results (that is, \cite[Theorem 3.5]{GarPe2014}) is directly jeopardized by the mistake introduced after \cite[Lemma 2.1]{GarPe2014}. The concrete obstacle appears because by applying the wrong identity commented above, we identify every commutative unital real C$^*$-algebra with a real function algebra of the form $C_r(K;F)$. More concretely, let $T:C(K_1)^{\tau_1} \to C(K_2)^{\tau_2}$ be an orthogonality preserving linear mapping, and let $\sigma_i : K_i\to K_i$ be a topological involution satisfying $\tau_i (f) = \overline{f\circ \sigma_i}$. According to what is claimed in \cite[page 287]{GarPe2014}:

\emph{``$\dots$ keeping in mind the notation in the previous section, we write $L_i := \mathcal{O}_i\cup F_i,$ where $\mathcal{O}_i$ and $F_i=\{t\in K_i : \sigma_i (t) =t\}$ are the subsets of $K_i$ given by Lemma \cite[Lemma 2.1]{GarPe2014}. The map sending each $f$ in  $C(K_i)^{\tau_i}$ to its restriction to $L_i$ is a C$^*$-isomorphism (and hence a surjective linear isometry) from $C(K_i)^{\tau_i}$ onto the real C$^*$-algebra $C_r(L_i)$ of all continuous functions $f : L_i \to \mathbb{C}$ taking real values on $F_i$. Thus, studying orthogonality preserving linear maps between $C(K)^{\tau}$ spaces is equivalent to study orthogonality preserving linear mappings between the corresponding $C_r(L)$-spaces.''}

Unfortunately, since, in general, $K_i\neq \mathcal{O}_i\cup F_i \cup \sigma(\mathcal{O}_i)$, we cannot guarantee that $L_i := \mathcal{O}_i\cup F_i,$ is a closed subset of $K_i$ (compare the example given in page \pageref{page 1}). Furthermore, there are examples of $C(K)^{\tau}$-spaces which are not real $C^*$-isomorphic to a real C$^*$-algebra of the form $C_r(X,F)$ (compare \ref{remark C(K)tau not isomorphic to C(K,f)}). In summary, the proof of \cite[Theorem 3.5]{GarPe2014} is only valid for orthogonality preserving linear bijections between commutative unital real C$^*$-algebras which are of the form $C_r(K,F)$.\smallskip

Though the main results in \cite{GarPe2014} remain valid in the form they are stated, the gap commented above makes invalid some of the arguments given in that paper. It is necessary to provide a complete and correct argument, which fix all the problems. We present here a complete revision of these results with new and complete arguments, which allow us to solve and fix the problems caused by the gap in the original proof. The problems caused in \cite[\S 2]{GarPe2014} are easily fixable and we shall just comment briefly the necessary changes in Section 2 here. However, the difficulties caused in the proof of the result asserting the automatic continuity of every orthogonality preserving linear bijection between commutative unital real C$^*$-algebras force us to present a more detailed revision in Section 3.

\section{Orthogonal forms on real C$^*$-algebras}\label{sec:1}

We revise in this section the proof of \cite[Theorem 2.4]{GarPe2014}. For conciseness reasons, we keep the notation in \cite{GarPe2014} without inserting explicit definitions. We shall confine ourselves to state the minimum changes necessary to fix the difficulties in \cite[\S 2]{GarPe2014}.\smallskip

Lemma 2.2 in \cite{GarPe2014} should be rewritten as follows:
\setcounter{section}{2}
\setcounter{theorem}{1}

\begin{lemma}\label{l spectral resolution for skew symmetric} In the notation of Lemma 2.1, let $B(A)= B(K)^{\tau}$, let $a\in B(K)^{\tau}_{sa},$ and let $b$ be an element in $B(A)_{skew}$. Then the following statements hold:\begin{enumerate}[$a)$]
\item  $b|F = 0$;
\item For each $\varepsilon >0$, there exist mutually disjoint Borel
sets $B_{1},\ldots,B_m$ and real numbers $\lambda_1,\ldots,\lambda_m$ satisfying $\sigma(B_j) \cap B_j=\emptyset$ and
$\displaystyle{\left\| b-\sum_{j=1}^{m}  i  \ \lambda_j (\chi_{_{B_j}} -  \chi_{_{\sigma(B_j)}})\right\| <\varepsilon;}$
\item For each $\varepsilon >0$, there exist mutually disjoint Borel
sets $C_{1},\ldots,C_m \subset K$ and real numbers $\mu_1,\ldots,\mu_m$ satisfying $\sigma (C_j)= C_j$ and
$\displaystyle{\left\| a-\sum_{j=1}^{m} \mu_j \chi_{_{C_j}}\right\| <\varepsilon.}$
\end{enumerate}
\end{lemma}

The arguments given in \cite[Comments before Lemma 2.1 and Proof of Lemma 2.2]{GarPe2014} remain valid here.\smallskip

The notation in \cite[Lemma 2.3]{GarPe2014} should be replaced with the following: For each $C\subseteq K$ with $C\cap \sigma (C)=\emptyset$ we shall write $u_{_C}= i  \ (\chi_{_C}-\chi_{_{\sigma(C)}}).$ The symbol $u_{_0}$ will stand for the projection $\chi_{_{K\backslash F}}.$ Clearly,
$1=\chi_{_F}+ u_{_0}$ where $1$ is the unit element in $B(K)^{\tau}.$ By Lemma \ref{l spectral resolution for skew symmetric} $a)$, for each  $b\in B(K)^{\tau}_{skew}$ we have $b \perp \chi_{_F},$ and so $b = b u_{_0}.$ The statement of \cite[Proposition 2.3]{GarPe2014} should be modified in the following sense:

\begin{proposition}\label{p extn Boral algebra orthogonal}
Let $K$ be a compact Hausdorff space, $\tau$
a period-2 conjugate-linear isometric $^*$-homomorphism on $C(K)$,
$A= C(K)^{\tau}$, and $V: A \times A \to \mathbb{R}$ be an orthogonal bounded
bilinear form whose Arens extension is denoted by $V^{**} :
A^{**}\times A^{**} \to \mathbb{R}$. Let $\sigma: K \to K$ be a period-2 homeomorphism
satisfying $\tau (a) (t) =  \overline{a (\sigma (t))},$ for all $t\in K$, $a\in C (K)$.
Then the following assertions hold for all Borel subsets $D,B,C$ of $K$ with
$\sigma(B)\cap B=\sigma(C) \cap C = \emptyset$ and $\sigma(D) =D$:
\begin{enumerate}[$a)$]
\item $V(\chi_{_D},u_{_B})=V(u_{_B},\chi_{_D})=0,$ whenever $D\cap
B=\emptyset$;

\item $V(u_{_B},u_{_C})=0,$ whenever $B\cap C=\emptyset;$

\item
$V((u_{_0}-u_{_C}u^*_{_C})u_{_B},u_{_C})=V(u_{_C},(u_{_0}-u_{_C}u^*_{_C})u_{_B})=0.$
\end{enumerate}
\end{proposition}

The proof given in \cite{GarPe2014} remains valid with the obvious changes in the notation. We include here an sketch of the changes for completeness reasons.

\begin{proof} By an abuse of notation, we identify $V$ and $V^{**}$. Arguing as in the first part of the proof of \cite[Proposition 2.3]{GarPe2014} we get

\begin{equation}\label{eq orth sym vs anti compact} V\left(u_{_{K_2}},\frac{1}{2}(\chi_{_{K_1}}+\chi_{_{\sigma(K_1)}})\right)=V\left(\frac{1}{2}(\chi_{_{K_1}}+\chi_{_{\sigma(K_1)}}), u_{_{K_2}}\right)=0, \end{equation}
and
\begin{equation}\label{eq orth anti vs anti compact} V\left(u_{_{K_1}},u_{_{K_2}}\right)=0,
 \end{equation} whenever $K_1$ and $K_2$ are two compact subsets of $K$ such that $K_1, K_2, \sigma(K_1)$ and $\sigma(K_2)$ are pairwise disjoint.\smallskip

$a)$ Let now $D,B$ be two disjoint Borel subsets of $K$ such that $\sigma(D)=D$ and $B\cap \sigma(B) =\emptyset .$
By inner regularity there exist nets of the form $(\chi_{_{K_{\lambda}^{^D}}})_{\lambda}$ and $(\chi_{_{K_{\gamma}^{^B}}})_{\gamma}$ such that $(\chi_{_{K_{\lambda}^{^D}}})_{\lambda}$ and  $(\chi_{_{K_{\gamma}^{^B}}})_{\gamma}$ converge in the weak$^*$ topology of $C(K)^{**}$ to $\chi_{_D}$ and $\chi_{_B}$, respectively, where each $K_{\lambda}^{^D} \subseteq D$ and each $K_{\gamma}^{^B}\subseteq B$ is a
compact subset of $K$. By the assumptions made on $D$ and $B$ we have that $K_{\lambda}^{^D}\cap K_{\gamma}^{^B}=K_{\lambda}^{^D}\cap \sigma(K_{\gamma}^{^B})=\emptyset$ and $K_{\gamma}^{^B}\cap \sigma(K_{\gamma}^{^B})=\emptyset$ for all $\lambda$ and $\gamma.$ By (\ref{eq orth sym vs anti compact}) and the separate weak$^*$ continuity of $V$ we have \begin{equation}
 \label{eq orth symm vs anti} V(\chi_{_D},u_{_B})=w^*-\lim_{\lambda} \left(w^*-\lim_{\gamma} V\left(\frac{\chi_{_{K_{\lambda}^{^D}}}+\chi_{_{\sigma(K_{\lambda}^{^D})}}}{2},
  u_{_{K_{\gamma}^{^B}}}\right)\right)=0,
 \end{equation}
 and
\begin{equation}
 \label{eq orth antivs symm} V(u_{_B},\chi_{_{_D}})=0.
 \end{equation}\medskip

A similar argument, but replacing (\ref{eq orth sym vs anti compact}) with (\ref{eq orth anti vs anti compact}), applies to obtain $b)$.\smallskip

To prove the last statement, we observe that
$$(u_{_0}-u_cu_c^*)u_{_B}=(\chi_{_{K\backslash F}} -\chi_{_C}-\chi_{_{\sigma(C)}})u_{_B}=
  \chi_{_{K\setminus (F\cup C\cup \sigma(C))}} u_{_B}=u_{_{(K\setminus (C\cup \sigma(C)))\cap
  B}},$$ and hence the statement $c)$ follows from $b)$.
\end{proof}

The statement of Theorem 2.4 in \cite{GarPe2014} remains unaltered, however, the proof of this theorem needs a slight modification from line 11.

\begin{proof}[Proof of Theorem \ref{t orhtogonal forms real abelian}] 
By Proposition \cite[Proposition 1.5]{GarPe2014} we have \begin{equation}
\label{eq theorem -1} V(a_1,a_2)=V(a_1 a_2,1),
\end{equation} for every $a_1,a_2$ in $B(K)^{\tau}_{sa}$.

To deal with the skew-symmetric part, let $D,B,C$ be Borel subsets of $K$ with, $D=\sigma(D)$, $B\cap \sigma(B)=\emptyset$ and $C\cap \sigma(C)=\emptyset$. From Proposition \ref{p extn Boral algebra orthogonal} $a)$, we have \begin{equation}\label{eq sym against anti}V(\chi_{_D},u_{_B})=V(\chi_{_D},u_{_B}(1-\chi_{_D}+\chi_{_D}))= V(\chi_{_D},u_{_{B\cap (K\backslash D)}})+ V(\chi_{_D},u_{_B}\chi_{_D})\end{equation} $$=V(\chi_{_D}-1+1,u_{_B}\chi_{_D})= V(-\chi_{_{(K\backslash D)}}+1,u_{_{(B\cap D)}}) =V(1,u_{_B}\chi_{_D}),$$
 and similarly,
\begin{equation}\label{eq anti against sym}V(u_{_B},\chi_{_D})=V(u_{_B}\chi_{_D},1).
\end{equation}

If we apply Proposition \ref{p extn Boral algebra orthogonal} $b)$ and $c)$, repeatedly, we deduce that \begin{equation} \label{eq anti anti 1} V(u_{_B},u_{_C})=V(u_{_B}u_{_0} ,u_{_C})=V(u_{_B}(u_{_0} +u_{_C}u^*_{_C} -u_{_C}u^*_{_C}),u_{_C})\end{equation} $$=   V(u_{_B}u_{_C}u^*_{_C},u_{_C})=V(u_{_B}u_{_C}u^*_{_C},u_{_C}-u_{_0}+u_{_0})$$ $$ = V(u_{_{(B\cap C)}},-u_{_{((K\backslash F)\backslash C)}}+u_{_0}) =V(u_{_{(B\cap C)}},u_{_0})=V(u_{_B}u_{_C},u_{_0}).$$ and similarly \begin{equation} \label{eq anti anti 2}V(u_{_B},u_{_C})= V(u_{_0},u_{_B}u_{_C}).\end{equation}

Let $\displaystyle{a_l=\sum_{j=1}^{m_l} \mu_{l,j} \chi_{_{D^l_j}},}$ $\displaystyle{b_l=\sum_{k=1}^{p_l} \lambda_{l,k} u_{_{B^l_k}}}$ ($l\in\{1,2\}$) be two simple elements in $B(K)^{\tau}_{sa}$ and $B(K)^{\tau}_{skew}$, respectively, where $\lambda_{l,k}, \mu_{l,j}\in \mathbb{R},$ for each $l\in \{1,2\},$ $\{D^l_1,\ldots, D^l_{m_l}\}$ and $\{B^l_1,\ldots, B^l_{p_l}\}$ are families of mutually disjoint Borel subsets of $K$ with $\sigma(D^l_j) = D^l_j$ and $B^l_i \cap \sigma(B^l_i) =\emptyset.$ By $(\ref{eq theorem -1})$, $(\ref{eq sym against anti})$, $(\ref{eq anti against sym})$, and $(\ref{eq anti anti 1}),$ we have
$$V(a_1+b_1, a_2+b_2)=V(a_1 a_2,1)+ \sum_{j=1}^{m_1} \sum_{k=1}^{p_2} \mu_{1,j}  \lambda_{2,k} V\left( \chi_{_{D^1_j}},  u_{_{B^2_k}}\right)$$ $$+\sum_{k=1}^{p_1}\sum_{j=1}^{m_2} \mu_{2,j}  \lambda_{1,k} V\left( u_{_{B^1_k}}, \chi_{_{D^2_j}}\right)+\sum_{k=1}^{p_1}\sum_{k=1}^{p_2} \lambda_{2,k} \lambda_{1,k} V\left( u_{_{B^1_k}},  u_{_{B^2_k}}\right)$$ $$= V(a_1 a_2,1)+ \sum_{j=1}^{m_1} \sum_{k=1}^{p_2} \mu_{1,j}  \lambda_{2,k} V\left(1 ,  \chi_{_{D^1_j}} u_{_{B^2_k}}\right)$$ $$+\sum_{k=1}^{p_1}\sum_{j=1}^{m_2} \mu_{2,j}  \lambda_{1,k} V\left( u_{_{B^1_k}} \chi_{_{D^2_j}} , 1\right)+\sum_{k=1}^{p_1}\sum_{k=1}^{p_2} \lambda_{2,k} \lambda_{1,k} V\left( u_{_{B^1_k}}u_{_{B^2_k}} , u_{_0} \right) $$ $$= V(a_1 a_2,1)+  V\left(1 ,  a_1 b_2\right) +V\left( b_1 a_2,1\right) + V\left( b_1 b_2 , u_{_0}\right)$$
$$= \psi_1 (a_1 a_2)+  \psi_2 \left(a_1 b_2\right) +\psi_1 \left( b_1 a_2 \right) + \psi_4 \left( b_1 b_2\right),$$
where $\psi_1,\psi_2,$ and $\psi_4$ are the functionals in $A^*$ defined by $\psi_1(x)= V(x,1),$ $\psi_2(x)= V(1,x),$ and $\psi_4(x) = V(x  ,u_{_0}),$ respectively. Since, by Lemma \ref{l spectral resolution for skew symmetric}, simple elements of the above form are norm-dense in $B(K)^{\tau}_{sa}$ and $B(K)^{\tau}_{skew}$, respectively, and $V$ is continuous, we deduce that $$V(a_1+b_1, a_2+b_2)= \psi_1 (a_1 a_2)+  \psi_2 \left(a_1 b_2\right) +\psi_1 \left( b_1 a_2 \right) + \psi_4 \left( b_1 b_2\right),$$ for every $a_1,a_2\in B(K)^{\tau}_{sa}$, $b_1,b_2 \in B(K)^{\tau}_{skew}.$\smallskip

Now, taking $\phi_1 = \frac14 (2 \psi_1 +\psi_2+\psi_4),$ $\phi_2 = \frac14 (2 \psi_1 -\psi_2-\psi_4),$
$\phi_3 = \frac14 ( \psi_2- \psi_4),$ and $\phi_4 = \frac14 (\psi_4-\psi_2),$ we get $$V(a_1+b_1, a_2+b_2)= \phi_1 ((a_1+b_1) (a_2+b_2))+  \phi_2 \left((a_1+b_1) (a_2+b_2)^*\right) $$ $$+\phi_3 \left( (a_1+b_1)^* (a_2+b_2) \right) + \phi_4 \left( (a_1+b_1)^* (a_2+b_2)^*\right),$$ for every $a_1,a_2\in B(K)^{\tau}_{sa}$, $b_1,b_2 \in B(K)^{\tau}_{skew}.$\smallskip

Finally, defining $\varphi_1 (x) = \phi_1 (x)+ \phi_4(x^*)$ and  $\varphi_2(x)= \phi_2(x)+ \phi_3(x^*)$ $(x\in A)$, we get the desired statement.
\end{proof}

\section{Orthogonality preservers between $C(K)^{\tau}$-spaces}

In this section we shall study orthogonality preserving linear bijections between commutative unital real C$^*$-algebras. The aim is to provide to the reader an argument to avoid the difficulties in the proof of \cite[Theorem 3.5]{GarPe2014}. We have already commented in the introduction that the arguments in the proof of \cite[Theorem 3.5]{GarPe2014} are only valid to show that every orthogonality preserving linear bijection between $C_r(K;F)$-spaces is continuous (compare also page 287 and section 3 in the same paper).\smallskip

We begin this section with a remark that present a commutative unital real C$^*$-algebra which is not C$^*$-isomorphic to a real function algebra of the form $C_r(K, F)$.

\begin{remark}\label{remark C(K)tau not isomorphic to C(K,f)}{\rm Let $K=\{t_1,t_2\}$ equipped with the discrete topology, $\sigma : K\to K$ the topological involution given by $\sigma (t_1) = t_2$. It is easy to check that $C(K)^\tau \equiv \mathbb{C}_{_{\mathbb{R}}}$, the complex field regarded as a real space. Suppose there exists a compact Hausdorff space $X$ and a closed subset $F\subseteq X$ such that $C(K)^\tau \equiv \mathbb{C}_{_{\mathbb{R}}}$ is C$^*$-isomorphic to $C_r(X;F)$. Since $C(X,\mathbb{R})$ is a real subspace of $C_r(X;F)$, we easily deduce from Urysohn's lemma that $\sharp X \leq 2$ and hence $X= \{s_1,s_2\}$. In this case, there are only three possibilities to consider, namely, $F=\emptyset$, $F=\{s_2\}$ and $F= X$. The real C$^*$-algebra $C_r (X,F)$ coincides with $C(K)= \mathbb{C}\oplus^{\infty} \mathbb{C}$, $\mathbb{C}\oplus^{\infty} \mathbb{R}$ and $\mathbb{R}\oplus^{\infty} \mathbb{R}$, respectively. None of the above real C$^*$-algebras is C$^*$-isomorphic to $C(K)^\tau\equiv \mathbb{C}_{_{\mathbb{R}}}$.}\end{remark}

The above Remark \ref{remark C(K)tau not isomorphic to C(K,f)} implies that we cannot derived that every orthogonality preserving linear bijection between commutative unital real C$^*$-algebras is (automatically) continuous as a consequence of \cite[Theorem 3.2, Remark 3.4 and the proof of Theorem 3.5]{GarPe2014}.\smallskip

Henceforth, let $T: C(K_1)^{\tau_1}\to C(K_2)^{\tau_2}$ be an orthogonality preserving real linear bijection\hyphenation{bijection}. Following standard notation, for each $s\in K_2$, we denote by $\delta_s : C(K_2)^{\tau_2} \to \mathbb{C}$ the linear mapping given by $\delta_s (g)= g(s)$ ($g\in C(K_2)^{\tau_2}$). We observe that $T$ being surjective implies that $\delta_s T: C(K_1)^{\tau_1} \to \mathbb{C}$ is a non-zero linear map. The symbol $\hbox{supp} (\delta_s T)$ will denote the set of all $t\in K_1$ such that for each open set $U= \sigma_1 (U)\subseteq K_1$ with  $t\in U $ there exists $f\in C(K_1)^{\tau_1}$ with  $\hbox{coz}(f)\subseteq U$ and $\delta_{s} (T(f)) \neq 0$. Let us recall that the \emph{cozero} set, $coz(f)$, of a function $f\in C(K_1)^{\tau_1}$ is the set $\{t\in K_1 : f(t)\neq 0\}$. The equality $\sigma_1 (coz(f)) =coz (f)$ holds for every $f\in C(K_1)^{\tau_1}$.

\begin{proposition}\label{l tech 1}
\begin{enumerate}[$(a)$]
\item For each $s\in K_2$ there exists a unique element $t_{s}\in K_1$ such that the set $\hbox{supp}(\delta_s T) = \{t_{s},\sigma_1(t_s)\}$;

\item For every $s\in K_2$, we have $\delta_s T$ is continuous if and only if $\delta_{\sigma_2(s)} T$  is continuous. Moreover, the equality $\hbox{supp} (\delta_s T)$ $= \hbox{supp} (\delta_{\sigma_2(s)} T)$ holds for every $s\in K_2$.
\end{enumerate}
\end{proposition}

\begin{proof} $a)$ Let us take $s\in K_2$. We first show that  $\hbox{supp} (\delta_s T)$ contains at most two points of the form $t$ and $\sigma_1 (t)$. Arguing by contradiction, we assume that there exist $t_1, t_2$ in $\hbox{supp} (\delta_s T)$ with $t_1\neq t_2, \sigma_1(t_2)$. In this case, we can find two open disjoint subsets $U_1$ and $U_2$ in $K_1$ with $ t_i \in  U_i =\sigma_i (U_i),$ and two elements $f_1,f_2\in C(K_1)^{\tau_1}$ satisfying $\hbox{coz} (f_i)\subseteq U_i$ and $\delta_s T(f_i)\neq 0,$ for every $i=1,2.$ This is impossible because $f_1\perp f_2$ and $T$ is orthogonality preserving.\smallskip

We shall show next that $\hbox{supp} (\delta_s T)\neq \emptyset$. Otherwise, $\hbox{supp} (\delta_s T)= \emptyset$. Then, for each $t\in K_1$, there exists an open subset $U_t= \sigma_1 (U_t)$ with $t\in U_t$ and $\delta_s (T(f)) = 0$ for every $f\in C(K_1)^{\tau_1}$ with $\hbox{coz} (f)\subseteq U_t$. By a compactness argument, we can find a finite open cover $\{U_1,\ldots,U_m\}$ of $K_1$ satisfying $U_k= \sigma_1 (U_k)$ $(\forall k)$ and $\delta_s (T(f)) = 0$ for every $f\in C(K_1)^{\tau_1}$ with $\hbox{coz} (f)\subseteq U_k$ for some $k$. Let $g_1,\ldots,g_m$ be a continuous decomposition of the identity in $C(K_1)$ subordinate to  $U_1,\ldots,U_m$. Since $U_k= \sigma_1 (U_k)$, the elements $f_1=\frac{g_1+\tau_1 (g_1)}{2},\ldots, f_m=\frac{g_m+\tau_1 (g_m)}{2}$ define a continuous decomposition of the unit in $C(K_1)^{\tau_1}$ subordinate to  $U_1,\ldots,U_m$. For each $f\in C(K_1)^{\tau_1}$ we have $$\delta_s T(f) = \sum_{k=1}^{m} \delta_s T(f f_i) = 0,$$ which contradicts $\delta_s T \neq 0.$

$b)$ Let $s\in K_2.$  Clearly, $\delta_s (g) =  \overline{\delta_{\sigma_2(s)} (g)},$  for every $g\in C(K_2)^{\tau_2}.$ The first statement follows from the identity $\delta_s T =\overline{\delta_{\sigma_2(s)} T}.$ Let us assume that $\hbox{supp} (\delta_s T)\neq \hbox{supp} (\delta_{\sigma_2(s)} T).$ In this case there exist $t_1\in \hbox{supp} (\delta_s T)$ and $t_2\in \hbox{supp} (\delta_{\sigma_2(s)} T)$ with $t_1\neq t_2 ,\sigma_1 (t_2).$ We can find two open disjoint subsets $U_1$ and $U_2$ satisfying $t_i \in U_i = \sigma_1 (U_i)$ and two elements $f_1,f_2\in C(K_1)^{\tau_1}$ with coz$(f_i)\subseteq U_i$, $\delta_{s} T (f_1)\neq 0$ and $\delta_{\sigma_2(s)} T(f_2) = \overline{\delta_{s}T(f_2)} \neq 0,$ which contradicts $T(f_1)\perp T(f_2)$.
\end{proof}

Let us define an equivalence relation on $K_i$ given by $t\sim s$ if $\sigma_i (t) = \sigma_i(s)$. It is known that the quotient space $[K_i]= K_i/\sim$ is compact. It is not hard to check that $[K_i]= K_i/\sim$ is a compact Hausdorff space. The equivalence class of an element $t\in K_i$ is denoted by $[t]=\{t'\in K_i: t'\sim t \}$. Applying Proposition \ref{l tech 1}$(a)$, we can define a map $\varphi: K_2 \to [K_1]$, $s\mapsto [t]=\hbox{supp}(\delta_s T)$. By Proposition \ref{l tech 1}$(b)$, $\hbox{supp}(\delta_{s_1} T) =\hbox{supp}(\delta_{s_2} T)$, for every $s_1$, $s_2$ in $K_2$ with $s_1\sim s_2$. Therefore, the mapping $[\varphi]: [K_2] \to [K_1]$, $[s]\mapsto [t]=\hbox{supp}(\delta_s T)$ is well defined.\smallskip

\begin{lemma}\label{l varphi not in supp(f)} Let $s$ be an element in $K_2,$ then $\delta_s T (f) =0$ for every $f\in C(K_1)^{\tau_1}$ with $\hbox{supp}(\delta_s T)\cap \overline{\hbox{coz}(f)} =\emptyset$. In particular, the set $\displaystyle \mathcal{S}\hbox{upp}(K_2)=\bigcup_{s\in K_2} \hbox{supp}(\delta_s T)$ is dense in $K_1$.
\end{lemma}

\begin{proof} Suppose we have $f\in C(K_1)^{\tau_1}$ with $\hbox{supp}(\delta_s T)\cap \overline{\hbox{coz}(f)} =\emptyset.$ We can find an open set $U\subseteq K_1$ with $\hbox{supp}(\delta_s T)\subseteq U= \sigma_1 (U)$ and $U\cap \overline{\hbox{coz}(f)} =\emptyset.$ By assumptions, there exists $g\in C(K_1)^{\tau_1}$ with coz$(g)\subset U$ and $\delta_{s} T (g) \neq 0.$ $T$ being orthogonality preserving and $f\perp g$ imply that $T(f)\perp T(g),$ and hence $\delta_s T (f) =0$.\smallskip

For the second statement, suppose we can find $t_0\in K_1\backslash \overline{\left(\displaystyle \bigcup_{s\in K_2} \hbox{supp}(\delta_s T)\right)}$. We observe that $\displaystyle \overline{\bigcup_{s\in K_2} \hbox{supp}(\delta_s T)}$ is $\sigma_1$-symmetric. Then, by Urysohn's lemma, there exists $f_0\in C(K_1)^{\tau_1}$ such that $0\leq f_0\leq 1$, $f_0 (t_0)=1$ and coz$(f_0)\cap \overline{\bigcup_{s\in K_2} \hbox{supp}(\delta_s T)}=\emptyset$. We deduce from the first part of the lemma that $T(f_0) (s)=0$, for every $s\in K_2$, which contradicts the injectivity of $T$.
\end{proof}

We shall derive next some consequences of the previous results.

\begin{proposition}\label{l density and kernels} Let $s$ be an element in $K_2$ such that $\delta_s T : C(K_1)^{\tau_1} \to \mathbb{C}$ is a continuous linear map. Then, for each $t\in [\varphi] [s] = \hbox{supp}(\delta_s T)$, there exist $\lambda_s,\mu_s\in \mathbb{C}$ such that $$\delta_s T (f) = \lambda_s \Re\hbox{e}\delta_{t} (f) + {\mu_s} \Im\hbox{m}\delta_t (f),$$ for every $f\in C(K_1)^{\tau_1}$. Moreover, $\lambda_s$ is unique for every $s$, while $\mu_s$ is unique whenever supp$(\delta_s T)$ contains two points {\rm(}i.e. $\sigma_1(t_s)\neq t_s${\rm)}. It is also clear that $\lambda_s = T(1) (s)$, for every $s$ as above.
\end{proposition}

\begin{proof} We already know that $\hbox{supp}(\delta_s T)=\{t_s,\sigma_1(t_s)\}$, for a unique $t_s\in K_1$. We fix this $t_s$.\smallskip

Let us consider the sets $J_{_{{supp}(\delta_s T)}} :=\{ f\in C(K_1)^{\tau_1} :  \hbox{supp}(\delta_s T) \cap \overline{\hbox{coz}(f)} =\emptyset\}$ and $K_{_{{supp}(\delta_s T)}} :=\{ f\in C(K_1)^{\tau_1} :  f|_{_{\hbox{supp}(\delta_s T)}} = 0\}$. Clearly, $J_{_{{supp}(\delta_s T)}}\subseteq K_{_{{supp}(\delta_s T)}}$. The arguments given by K. Jarosz in \cite[141]{Jar} remain valid to show that $J_{_{{supp}(\delta_s T)}}$ is norm-dense in $K_{_{{supp}(\delta_s T)}}$. Lemma \ref{l varphi not in supp(f)} implies that $J_{_{{supp}(\delta_s T)}} \subseteq \ker (\delta_s T)= \ker (\delta_{_{\sigma_2(s)}} T)$. We deduce from the continuity of $\delta_s T$ that \begin{equation}\label{eq kernels} \ker (\delta_{t_s})=\ker (\delta_{_{\sigma_1(t_s)}})= K_{_{{supp}(\delta_s T)}} \subseteq \ker (\delta_s T)= \ker (\delta_{_{\sigma_2(s)}} T).
\end{equation}

The real linear functionals $\Re\hbox{e} \delta_s T$, $\Im\hbox{m} \delta_s T$, $\Re\hbox{e} \delta_{t_s}$, and  $\Im\hbox{m} \delta_{t_s}$ are all continuous. Since $\ker(\Re\hbox{e} \delta_{t_s}) \cap \ker(\Im\hbox{m} \delta_{t_s}) = \ker(\delta_{t_s})$ and $\ker (\delta_s T) = \ker (\Re\hbox{e}\delta_s T)\cap \ker (\Im\hbox{m}\delta_s T)$, we deduce from \eqref{eq kernels} the existence of $\alpha_1,\alpha_2,\beta_1,\beta_2\in\mathbb{R}$ satisfying $$\Re\hbox{e}\delta_s T =\alpha_1 \Re\hbox{e} \delta_{t_s}+\alpha_2 \Im\hbox{m} \delta_{t_s}$$ and $$\Im\hbox{m}\delta_s T =\beta_1 \Re\hbox{e} \delta_{t_s} +\beta_2 \Im\hbox{m} \delta_{t_s}.$$ Taking $\lambda_s=\alpha_1+i \beta_1$ and $\mu_s=\alpha_2+i \beta_2$ we obtain $$\delta_s T = \lambda_s \Re\hbox{e}\delta_{t_s} + {\mu_s} \Im\hbox{m}\delta_{t_s}.$$

To prove the second statement suppose there are $\lambda_1,\lambda_2$, $\mu_1$ and $\mu_2$ in $\mathbb{C}$ such that $$\lambda_2 \Re\hbox{e}\delta_{t_s} + {\mu_2} \Im\hbox{m}\delta_{t_s}= \delta_s T = \lambda_1 \Re\hbox{e}\delta_{t_s} + {\mu_1} \Im\hbox{m}\delta_{t_s}.$$ Pick, via Urysohn's lemma, a function $f_0\in C(K_1)^{\tau_1}$ satisfying $f_0 (t_s)=1$. Then $\lambda_1 =\lambda_2$. When $\sigma(t_s)\neq t_s$, we can find another function $f_1\in C(K_1)^{\tau_1}$ satisfying $f_1 (t_s)=i$, and hence $\mu_1=\mu_2$.
\end{proof}

\begin{corollary}\label{c varphi injective} Under the above conditions, the following statements hold: \begin{enumerate}[$(a)$] \item Let $s$ be an element in $K_2$ such that $\delta_s T : C(K_1)^{\tau_1} \to \mathbb{C}$ is a continuous linear map and $\sigma_2 (s) \neq s$. Then $f(t_s)=0$, for every $t_s\in \hbox{supp}(\delta_{s} T)$;
\item Suppose $s_1,s_2$ are elements in $K_2$ such that $\sigma_2(s_j)\neq s_j$ for every $j=1,2$, $\delta_{s_1} T$ and $\delta_{s_2} T$ are continuous. If  $\hbox{supp}(\delta_{s_1} T) = \hbox{supp}(\delta_{s_2} T)$ then $s_1=s_2$ or $s_1=\sigma_2(s_2)$.
\end{enumerate}
\end{corollary}

\begin{proof}$(a)$ Let us fix $t_s\in \hbox{supp} (\delta_s T)$. By Proposition \ref{l density and kernels} there exist $\lambda_s,\mu_s\in \mathbb{C}$ such that \begin{equation}\label{eq expresion of deltasT} \delta_s T (f) = \lambda_s \Re\hbox{e}\delta_{t_s} (f) + {\mu_s} \Im\hbox{m}\delta_{t_s} (f),
 \end{equation}for every $f\in C(K_1)^{\tau_1}$. The surjectivity of $T$ implies that for each $\omega\in \mathbb{C}$ there are real numbers $\alpha,\beta$ such that $\omega = \lambda_s \alpha + {\mu_s} \beta$. This proves that $\{\lambda_s , {\mu_s} \}$ is a basis of the real space $\mathbb{C}_{\mathbb{R}}$. Therefore, by \eqref{eq expresion of deltasT}, $\delta_s T(f)=0$ implies $\Re\hbox{e}f(t_s)=\Im\hbox{m} f(t_s)=\delta_{t_s} (f) = f(t_s)=0,$ and also $f(\sigma_1 (t_s))=0$.\smallskip

$(b)$ Let us assume that $\hbox{supp}(\delta_{s_1} T) = \hbox{supp}(\delta_{s_2} T)$ for $s_1$ and $s_2$ as in the hypothesis. By $(a)$, $\delta_{s_1} T(f) =0$ implies $f (t_s) =0$, for every $t_s\in \hbox{supp}(\delta_{s_1} T)$, which, from \eqref{eq expresion of deltasT}, entails that $\delta_{s_2} T(f) =0$. Since $C(K_2)^{\tau_2}$ separates points $s_1,s_2$ in $K_2$ with $s_1 \nsim s_2$, the surjectivity of $T$ gives $s_1\sim s_2$.
\end{proof}

Given $s\in K_2$ we denote by $\|\delta_s T\|$ the norm of the linear mapping $\delta_s T: C(K_1)^{\tau_1}\to \mathbb{C}$ if the latter map is continuous, we set $\|\delta_s T\|=\infty$ otherwise. By Proposition \ref{l tech 1}$(b)$, $\|\delta_s T\|= \|\delta_{\sigma(s)} T\|$, for every $s\in K_2$.

\begin{lemma}\label{l norm deltasT local on open neighborhood} Let $s\in K_2$ and $t\in K_1$ such that $[t] = \hbox{supp}(\delta_{s} T)$. Let $U=\sigma_1 (U)\subset K_1$ be an open set satisfying $[t]\subset U$. The following statements hold:
\begin{enumerate}[$(a)$]\item If $\|\delta_s T\|<\infty$, then for each $\varepsilon>0$ there exists $f\in C(K_1)^{\tau_1}$ such that $\|f\|\leq 1$, coz$(f)\subset U$ and $|\delta_s T(f) | > \|\delta_s T\| -\varepsilon$;
\item If $\|\delta_s T\|=\infty$, then for each $R>0$ there exists $f\in C(K_1)^{\tau_1}$ such that $\|f\|\leq 1$, coz$(f)\subset U$ and $|\delta_s T(f) | > R$.
\end{enumerate}
\end{lemma}

\begin{proof}$(a)$ Let $g$ be an element in $C(K_1)^{\tau_1}$ such that $\|g\|\leq 1$ and $|\delta_s T(g) | > \|\delta_s T\| -\varepsilon$. Let $V\subseteq K_1$ be an open set satisfying $[t]\subset V\subseteq \overline{V}\subseteq U$. By Urysohn's lemma there exists $0\leq u\leq 1$ in $C(K_1)^{\tau_1}$ with $u|_{\overline{V}} \equiv 1$ and $u|_{K_1\backslash U}\equiv 0$. We observe that $t\notin \overline{\hbox{coz} (1-u)}$, and thus, Lemma \ref{l varphi not in supp(f)} implies that $\delta_s T( g (1-u)) =0$. We therefore have $$\delta_s T( g ) =\delta_s T( g (1-u) + gu ) =\delta_s T( g u),$$ which gives the desired statement for $f=gu\in C(K_1)^{\tau_1}$ with $\|gu \|\leq 1$ and coz$(gu)\subset U$.\smallskip

The proof of $(b)$ is very similar.
\end{proof}

\begin{lemma}\label{l boundedness one} Under the above assumptions, let $(t_n)$ and $(s_n)$ be sequences in $K_1$ and $K_2$, respectively, such that $[t_n]\neq [t_m]$, for every $n\neq m$ and $[t_n] = \hbox{supp}(\delta_{s_n} T)$ for every $n\in \mathbb{N}$. Then $\sup \{ \|\delta_{s_n} T\| : n\in \mathbb{N} \}<\infty$.
\end{lemma}

\begin{proof} Arguing by contradiction, we assume that, $\sup \{ \|\delta_{s_n} T\| : n\in \mathbb{N} \}=\infty.$ Up to an appropriate subsequence, we can find a sequence $(U_n)$ of mutually disjoint open subsets in $K_1$ such that $[t_n]\subset U_n =\sigma_1 (U_n)$ and $\sup \{ \|\delta_{s_n} T\| : n\in \mathbb{N} \}=\infty$.\smallskip

Since  $\sup \{ \|\delta_{s_n} T\| : n\in \mathbb{N} \}=\infty$, we apply Lemma \ref{l norm deltasT local on open neighborhood} to find a subsequence $(n_k)$ and sequence $(f_{k})_k\subset C(K_1)^{\tau_1}$ such that $\|f_{k}\|\leq 1$, coz$(f_{k})\subset U_{n_k}$ and $|\delta_{s_{n_k}} T (f_{k})| > 2^k$, for every natural $k$. The functions in the sequence $(f_{k})_k$ are mutually orthogonal, therefore $\displaystyle f_0 =\sum_{k=1}^{\infty} \frac{1}{k} f_{k}$ defines an element in $C(K_1)^{\tau_1}$. By orthogonality, for each $k_0\in \mathbb{N},$ we have $$|\delta_{s_{n_{k_0}}} T (f_{0})| = \Big|\delta_{s_{n_{k_0}}} T \Big(\frac{1}{k_{0}} f_{k_{0}} + \sum_{k=1, k\neq k_0}^{\infty} \frac{1}{k} f_{k}\Big)\Big| = \Big|\delta_{s_{n_{k_0}}} T \Big(\frac{1}{k_{0}} f_{k_{0}} \Big)\Big|  > \frac{2^{k_0}}{k_0},$$ which is impossible.
\end{proof}

Following a similar notation to that employed in \cite{Jar}, we set $Z_1 :=\{ s\in K_2 : \delta_{s} T \hbox{ is bounded }\}$ and $Z_2 :=\{ s\in K_2 : \delta_{s} T \hbox{ is unbounded }\}.$ Clearly $\sigma_2(Z_i) =Z_i$, for every $i=1,2$. The following conclusions can be straightforwardly derived from the previous results.

\begin{corollary}\label{c boundedness of varphi on points of discontinuity} \begin{enumerate}[$(a)$] \item The set $\displaystyle\mathcal{S}\hbox{upp}(Z_2):=\bigcup_{s\in Z_2} \hbox{supp}(\delta_{s} T) \subseteq K_1$ is finite and satisfies $\sigma_1(\mathcal{S}\hbox{upp}(Z_2)) =\mathcal{S}\hbox{upp}(Z_2)$. Furthermore, every point in $\mathcal{S}\hbox{upp}(Z_2)$ is non-isolated in $K_1$;
\item The set $\{ \|\delta_s T\| : s\in Z_1, \ \sigma_2 (s)\neq s \}$ is bounded.
\end{enumerate}
\end{corollary}

\begin{proof} $(a)$ The first statement follows from Lemma \ref{l boundedness one}. Suppose $t_0\in \mathcal{S}\hbox{upp}(Z_2)$ is isolated in $K_1$. We know that $t_0\in \hbox{supp} (\delta_s T)$ for certain $s\in Z_2$. In this case, we do not need the continuity assumptions in Proposition \ref{l density and kernels} to get a similar conclusion. Indeed, since $\{t_0\}$ is a closed subset, a function $f\in C(K_1)^{\tau_1}$ vanishes at $t_0$ if and only if $t_0\notin \overline{\hbox{coz} (f)}$. Moreover, $g_0=\delta_{t_0} + \delta_{\sigma_1(t_0)}\in C(K_1)^{\tau_1}$ and $f-f(t_0) g_0$ vanished at $t_0$, for every $f\in C(K_1)^{\tau_1}$. Lemma \ref{l varphi not in supp(f)} implies that $\delta_s T(f -f(t_0) g_0)=0$, and hence $\delta_s T(f) =\delta_s T(f(t_0) g_0)$ for every $f\in C(K_1)^{\tau_1}$. We note that the restriction of the real linear mapping $\delta_s T$ to the two dimensional subspace $\mathbb{C}_{\mathbb{R}} g_0$ is clearly continuous, and the same property holds for the linear mapping $\delta_{t_0} g_0 : C(K_1)^{\tau_1} \to \mathbb{C} g_0$, $f\mapsto \delta_{t_0} (f) g_0$. Therefore, the composed mapping $f\mapsto  \delta_s T(f(t_0) g_0) =\delta_s T (f)$ is linear and continuous, which contradicts that $s\in Z_2$.\smallskip

$(b)$ Arguing by contradiction, we assume that $\{ \|\delta_s T\| : s\in Z_1, \ \sigma_2 (s)\neq s \}$ is unbounded. We can find a sequence $(s_n)\subset Z_1$ such that $2^n \leq \|\delta_{s_n} T\|  \lneqq \|\delta_{s_{n+1}} T\|$. Clearly, $s_n\neq s_m$ and $s_n\neq \sigma_2(s_m)$, for every $n\neq m$ (we recall that $\|\delta_s T\|=\|\delta_{\sigma_2(s)} T\|$). Applying Corollary \ref{c varphi injective}$(b)$ we deduce that supp$(\delta_{s_n} T) \cap \hbox{supp} (\delta_{s_m} T)=\emptyset$. Finally, Lemma \ref{l boundedness one} gives the desired contradiction.\smallskip
\end{proof}

Let $\Delta_1 :=\{ (s,t) : s\in Z_1, \ t\in \hbox{supp}(\delta_s T) \}$. We define a mapping $\vartheta : \Delta_1 \to \mathbb{C}$ given by $\vartheta (s,t) =0$ if $\sigma_1(t)=t$ and $\vartheta (s,t) =\mu_s$ if $\sigma_1 (t)\neq t$, where $\mu_s$ is the unique element given by Proposition \ref{l density and kernels}.

\begin{proposition}\label{c boundedness of varphi on points of discontinuity 2} \begin{enumerate}[$(a)$]\item Let $(s,t)\in \Delta_1$ with $\sigma_1 (t)\neq t$, and let $g$ be any element in $C(K_1){\tau_1}$ satisfying $g|_{U} \equiv i$ for some open neighborhood $U$ of $t$.  Then $T(g) (s) =\mu_s =\vartheta (s,t)$.
\item The function $\vartheta$ is bounded;
\item The set $\{ \|\delta_s T\| : s\in Z_1\}$ is bounded;
\item $Z_1$ is closed.
\end{enumerate}
\end{proposition}

\begin{proof} $(a)$ Follows straightforwardly from Proposition \ref{l density and kernels}.\smallskip

$(b)$ If $\vartheta$ is unbounded, we can find a sequence $(s_n,t_n)\in \Delta_1$ satisfying $2^n < |\vartheta (s_n,t_n) |$ $<|\vartheta (s_{n+1},t_{n+1}) |.$ Obviously, $s_n\neq s_m$ and $s_n\neq \sigma_2(s_m)$, for every $n\neq m$ and $\hbox{supp} (\delta_{s_n} T) =\{t_n,\sigma_1(t_n)\}$ with $t_n\neq \sigma_1 (t_n)$. We claim that we can find a subsequence $(s_{n_k},t_{n_k})\subset \Delta_1$ satisfying
supp$(\delta_{s_{n_k}} T)\cap \hbox{supp} (\delta_{s_{n_m}} T)=\emptyset$ for every $k\neq m$. Let us assume, on the contrary, that there exists a natural $n_0$ such that supp$(\delta_{s_n} T)= \{t_0,\sigma_1 (t_0)\},$ for every $n\geq n_0$. Let $g$ be any element in $C(K_1){\tau_1}$ satisfying $g|_{U} \equiv i$ for some open neighborhood $U$ of $t_0$. It follows from $(a)$ that $2^n < |\vartheta (s_n,t_n)| =|T(g) (s_n)|,$ for every $n\geq n_0$, which contradicts that $T(g)\in C(K_2)^{\tau_2}.$ Having in mind that $\| \delta_{s_{n_k}} T\| \geq |\vartheta (s_{n_k},t_{n_k}) |$, the desired contradiction follows from Lemma \ref{l boundedness one}.\smallskip

$(c)$ We set $M_1=\sup \{|\vartheta (s,t_s)| : (s,t_s)\in \Delta_1\}$. Let $s$ be an element in $Z_1$, and let $t_s$ be an element in supp$(\delta_s T)$. By Proposition \ref{l density and kernels} the identity $$\delta_s T (f) = T(1) (s) \Re\hbox{e}\delta_{t_s} (f) + \vartheta (s,t_s) \Im\hbox{m}\delta_{t_s} (f),$$ holds for every $f\in C(K_1)^{\tau_1}$. Therefore, $|\delta_s T (f)| \leq \|T(1) \| + M_1$ for every $f\in C(K_1)^{\tau_1}$ with $\|f\|\leq 1$.\smallskip

$(d)$ Let $M=\sup \{\|\delta_s T\| : s\in Z_1 \}$. Fix an arbitrary $z_0\in \overline{Z_1}$ and a function $f\in C(K_1)^{\tau_1}$ with $\|f\|\leq 1$. We can find a net $(z_\mu)\subset Z_1$ converging to $z_0$ in the topology of $K_2$. Since $T(f)$ is a continuous function, we have $|T(f) (z_\mu)|\to |T(f) (z_0)|$. On the other hand, by $(d)$ we have $|T(f) (z_\mu)| = \leq M$, for every $\mu$. Therefore, $|\delta_{z_0} T(f) |=|T(f) (z_0)|\leq M$, for every $f\in C(K_1)^{\tau_1}$ with $\|f\|\leq 1$, which proves that $\delta_{z_0} T$ is bounded, and hence $z_0\in Z_1$.
\end{proof}

We state next the main result of this note, which provides an antidote to fill the gap we commented at the introduction.

\begin{theorem}\label{t automatic cont bijective OP} Every orthogonality preserving linear bijection between commutative unital real C$^*$-algebras is continuous.
\end{theorem}

\begin{proof} Let $T: C(K_1)^{\tau_1}\to C(K_2)^{\tau_2}$ be an orthogonality preserving linear bijection. We shall prove that $Z_2 :=\{ s\in K_2 : \delta_{s} T \hbox{ is unbounded }\}=\emptyset$. Suppose, on the contrary, that $Z_2\neq \emptyset$. Since the set $Z_1$ is closed (see Proposition \ref{c boundedness of varphi on points of discontinuity 2}$(d)$) the set $Z_2$ is an non-empty open subset of $K_2$ with $\sigma_2 (Z_2)=Z_2$. Find, via Urysohn's lemma, a non-zero function $g\in C(K_2)^{\tau_2}$ with coz$(g)\subseteq Z_2$.  By the surjectivity of $T$ there exists $0\neq h\in C(K_1)^{\tau_1}$ satisfying $T(h) = g$. According to the notation above, we set $\displaystyle \mathcal{S}\hbox{upp}(K_2) =\bigcup_{s\in K_2} \hbox{supp}(\delta_s T)$, $\displaystyle \mathcal{S}\hbox{upp}(Z_1) =\bigcup_{s\in Z_1} \hbox{supp}(\delta_s T)$, and $\displaystyle \mathcal{S}\hbox{upp}(Z_2) =\bigcup_{s\in Z_2} \hbox{supp}(\delta_s T)$. We claim that \begin{equation}\label{eq h vanishes on suppZ1} h(t) =0 \hbox{, for every } t\in \mathcal{S}\hbox{upp}(Z_1).
\end{equation}

To prove the claim, let $t_0$ be an element in $\mathcal{S}\hbox{upp}(Z_1)$. Since $\mathcal{S}\hbox{upp}(Z_2)$ is a $\sigma_1$-symmetric finite set in $K_1$ disjoint from $\mathcal{S}\hbox{upp}(Z_1)$, there exist disjoint open sets $U_1,U_2\subseteq K_1$ and a function $k\in C(K_1)^{\tau_1}$ such that $\sigma_1 (U_j)=U_j$ for every $j=1,2$, $t_0\in U_1$, $\mathcal{S}\hbox{upp}(Z_2) \subseteq U_2$, $k^* =k$ (i.e. $k(K_1)\subseteq \mathbb{R}$), $k(t_0)=1$, and coz$(k)\subseteq U_1$. We claim that $T(k h) =0.$ Indeed, for $s\in Z_2$, since $\mathcal{S}\hbox{upp}(Z_2)\cap \overline{\hbox{coz} (k)}=\emptyset$, Lemma \ref{l varphi not in supp(f)} implies that $T(k h) (s) = \delta_s T(kh) =0$. For each $s\in Z_1$, and $t_s\in \hbox{supp} (\delta_s T)$, $k(t_s)\in \mathbb{R}$ and hence $k h, k(t_s) h\in C(K_1)^{\tau_1}$ with $(k h) (t_s) =  (k(t_s) h) (t_s)$. By Proposition \ref{l density and kernels} we also have: $$\delta_s T (kh) = T(1) (s) \Re\hbox{e}\delta_{t_s} (k h) + \vartheta (s,t_s) \Im\hbox{m}\delta_{t_s} (kh),$$
$$=T(1) (s) \Re\hbox{e}\delta_{t_s} (k(t_s) h) + \vartheta (s,t_s) \Im\hbox{m}\delta_{t_s} (k(t_s) h) $$ $$= \delta_s T(k(t_s) h) = k(t_s) \delta_s T(h) =k(t_s) \delta_s (g) =0,$$ because $s\in Z_1$ and coz$(g)\subseteq Z_2$. This proves the second claim, and hence $T(k h) =0.$ The injectivity of $T$ implies $k h =0$ and hence $0=(k h) (t_0) = h(t_0)$, which completes the proof of the first claim.\smallskip

By Lemma \ref{l varphi not in supp(f)} the set $\displaystyle \mathcal{S}\hbox{upp}(K_2)$ is dense in $K_1$. The subset $\mathcal{S}\hbox{upp}(Z_2)$ is finite and every point in $\mathcal{S}\hbox{upp}(Z_2)$ is non-isolated in $K_1$ (compare Corollary \ref{c boundedness of varphi on points of discontinuity}$(a)$). Since $\mathcal{S}\hbox{upp}(K_2) = \mathcal{S}\hbox{upp}(Z_2) \stackrel{\circ}{\cup} \mathcal{S}\hbox{upp}(Z_1)$. It is not hard to see, from the normality of $K_1$, that $\mathcal{S}\hbox{upp}(Z_1)$ must be also dense in $K_1$. Since, by \eqref{eq h vanishes on suppZ1}, $h(t) =0$ for every $t\in \mathcal{S}\hbox{upp}(Z_1)$, we deduce from the continuity of $h$ that $h=0$, and hence $0=T(h) =g$, which gives the final contradiction.
\end{proof}

A generalized version of \cite[Corollary 3.10]{GarPe2014} for real C$^*$-algebras read as follows:

\begin{corollary}\label{c surjective bio isomorphism new} The following statements are equivalent: \begin{enumerate}[$(a)$]\item There exists a bi-orthogonality preserving linear bijection $T: C(K_1)^{\tau_1}\to C(K_2)^{\tau_2}$;
\item There exists a real C$^*$-isomorphism $S: C(K_1)^{\tau_1}\to C(K_2)^{\tau_2}$;
\item There exists a C$^*$-isomorphism $\widetilde{S}: C(K_1)\to C(K_2)$;
\item $K_1$ and $K_2$ are homeomorphic.$\hfill\Box$
\end{enumerate}
\end{corollary}

\end{document}